\documentclass[letterpaper, 11pt,  reqno]{amsart}

\usepackage{amsmath,amssymb,amscd,amsthm,amsxtra, esint}
\usepackage{tikz} 

\usepackage{comment}
\usepackage{mathtools}
\usepackage{color}
\usepackage[implicit=true]{hyperref}
\usepackage{scalerel} 
\usepackage{bbm}
\usepackage{todonotes}

\usepackage{soul}

\usepackage{cases}

\usepackage{enumitem}

\usepackage[margin=1.2in,marginparwidth=1.5cm, marginparsep=0.5cm]{geometry}

\allowdisplaybreaks[2]

\usepackage{color}

\definecolor{gr}{rgb}   {0.,   0.69,   0.23 }
\definecolor{bl}{rgb}   {0.,   0.5,   1. }
\definecolor{mg}{rgb}   {0.85,  0.,    0.85}
\definecolor{yl}{rgb}   {0.8,  0.7,   0.}
\definecolor{or}{rgb}  {0.7,0.2,0.2}

\usetikzlibrary{shapes.misc}
\usetikzlibrary{shapes.symbols}
\usetikzlibrary{shapes.geometric}

\tikzset{
	ddot/.style={circle,fill=white,draw=black,inner sep=0pt,minimum size=0.8mm},
	>=stealth,
	}

\tikzset{
	ddot2/.style={circle,fill=black,draw=black,inner sep=0pt,minimum size=0.8mm},
	>=stealth,
	}

\newtheorem{theorem}{Theorem} [section]

\newtheorem{lemma}[theorem]{Lemma}
\newtheorem{proposition}[theorem]{Proposition}
\newtheorem{remark}[theorem]{Remark}

\newtheorem{definition}[theorem]{Definition}
\newtheorem{corollary}[theorem]{Corollary}

\DeclareMathOperator{\med}{med}

\newcommand{\noi}{\noindent}
\newcommand{\Z}{\mathbb{Z}}
\newcommand{\R}{\mathbb{R}}
\newcommand{\C}{\mathbb{C}}
\newcommand{\T}{\mathbb{T}}

\newcommand{\nb}{\nabla}

\newcommand{\Dl}{\Delta}
\newcommand{\eps}{\varepsilon}

\newcommand{\G}{\Gamma}
\newcommand{\ld}{\lambda}

\newcommand{\ft}{\widehat}

\newcommand{\wt}{\widetilde}
\newcommand{\cj}{\overline}

\newcommand{\dt}{\partial_t}

\newcommand{\HS}{\textup{HS}}

\renewcommand{\l}{\ell}
\renewcommand{\o}{\omega}
\renewcommand{\O}{\Omega}

\newcommand{\les}{\lesssim}
\newcommand{\ges}{\gtrsim}

\newcommand{\jb}[1]
{\langle #1 \rangle}

\newcommand{\ind}{\mathbf 1}

\newcommand{\NN}{\mathcal{N}}

\newcommand{\pa}{\partial}

\newcommand{\N}{\mathbb{N}}

\newtheorem*{ackno}{Acknowledgements}

\newcommand{\Nb}{\mathbf{N}}

\numberwithin{equation}{section}
\numberwithin{theorem}{section}

\makeatletter
\@namedef{subjclassname@2020}{%
  \textup{2020} Mathematics Subject Classification}
\makeatother

\begin{document}
\baselineskip = 13.5pt
\baselineskip = 13.5pt

\title[GWP of the energy critical SNLS on $\T^3$]
{Global well-posedness of the energy-critical stochastic nonlinear Schr\"odinger equation
\\
on the three-dimensional torus}

\author[G.~Li, M.~Okamoto, and L.~Tao]
{Guopeng Li, Mamoru Okamoto, and Liying Tao}

\address{
Guopeng Li, School of Mathematics\\
The University of Edinburgh\\
and The Maxwell Institute for the Mathematical Sciences\\
James Clerk Maxwell Building\\
The King's Buildings\\
Peter Guthrie Tait Road\\
Edinburgh\\ 
EH9 3FD\\
United Kingdom, 
and
Department of Mathematics and Statistics, 
Beijing Institute of Technology, 
Beijing, China}
\email{guopeng.li@bit.edu.cn}

\address{
Mamoru Okamoto\\
Department of Mathematics\\
Graduate School of Advanced Science and 
Engineering, Hiroshima University\\
1-3-1 Kagamiyama, Higashi-Hiroshima, 739-8526\\ Japan}
\email{mokamoto@hiroshima-u.ac.jp}

\address{Liying Tao,
Graduate School of China Academy of Engineering Physics, Beijing, 100088, China}
\email{taoliying20@gscaep.ac.cn}

\begin{abstract}
We study the Cauchy problem of the defocusing energy-critical stochastic nonlinear Schr\"odinger equation (SNLS) on the three dimensional torus,
 forced by an additive noise.
We adapt the atomic spaces framework in the context of the energy-critical nonlinear Schr\"odinger equation,
and employ probabilistic perturbation arguments in the context of stochastic PDEs, 
establishing the global well-posedness of the defocusing energy-critical quintic SNLS in the energy space. 
It is the first global well-posedness result 
for the periodic SNLS in a critical space.

\end{abstract}



\maketitle

\tableofcontents

\section{Introduction}
\subsection{Stochastic nonlinear Schr\"odinger equation}
We consider the Cauchy problem for the stochastic nonlinear Schr\"{o}dinger equation (SNLS) with an additive noise:
\begin{equation}
\label{SNLS}
\begin{cases}
 i \dt u + \Dl  u = |u|^{4} u + \phi \xi \\
 u|_{t = 0} = u_0,
\end{cases}
\qquad (t, x) \in \R_+\times \T^3,
\end{equation}

\noi
where $\xi (t,x)$ denotes a space-time white noise on $\R_+\times \T^3$ and $\phi$ is a bounded operator on $L^2(\T^3)$. In this paper, we restrict our attention to the defocusing case.
Our main goal is to  establish the global well-posedness in $H^1(\T^3)$ of \eqref{SNLS}
via atomic spaces machinery.

Let us first go over the notion of the scaling-critical regularity
for the  (deterministic) defocusing nonlinear Schr\"odinger equation (NLS) on $\R^3$:
\begin{equation}
i \partial_t u +   \Delta u =  |u|^{4} u.
\label{NLS1}
\end{equation}

\noi
The NLS  \eqref{NLS1} on $\R^3$ is known to enjoy the following dilation symmetry:
\begin{equation*}
 u(t, x) \longmapsto u_\ld(t, x) = \ld^{\frac 12 } u (\ld^{2}t, \ld x)
\end{equation*}

\noi
for $\ld >0$. Namely, if $u$ is a solution to \eqref{NLS1} on $\R^3$, then the scaled function $u_\ld$ is also a solution to \eqref{NLS1} on $\R^3$.
A direct computation yields
\[
\| u(0) \|_{\dot H^1(\R^3)} =\| u_\ld (0) \|_{\dot H^1(\R^3)},
\]

\noi
which shows that the
scaling-critical Sobolev index is $ s_\text{crit} = 1$.
It is also easy to check that 
the energy
\[
\frac 12 \int_{\R^3} |\nb u(t, x)|^2 dx
+ \frac{1}{6} \int_{\R^3} |u(t, x)|^6 dx
\]

\noi
is invariant under 
this dilation symmetry.
For these reasons,
we call \eqref{NLS1} the energy critical NLS.
While there is no dilation symmetry on $\T^3$, 
we still refer to \eqref{NLS1} on $\T^3$ as energy critical.

The study of energy-critical NLS has always been a challenging problem. 
Especially on periodic domains, 
due to the lack of dispersion,
the well-posedness problem has been notably arduous, which requires much more delicate arguments. 
For the quintic NLS on $\T^3$, 
the  small data global well-posedness result in $H^1(\T^3)$ was done by 
Herr-Tataru-Tzvetkov \cite{HTT11},
where they made the first  energy-critical global well-posedness result in 
the setting of compact manifolds.
However, the removal of this smallness condition 
 is a remaining difficult issue. 
Still, after substantial efforts via
 profile decomposition, Ionescu-Pausader \cite{IP12}
 proved large data global well-posedness in $H^1(\T^3)$,
 in the defocusing case, and it was the first critical result for NLS on a compact manifold.
To see more deterministic energy-critical NLS in other settings in
\cite{BO99, CKSTT08, IP12b,KM06, KV10, KV12, RV07, V07, Y21}

In general, the study of SNLS at critical regularities has only a handful of results; see
\cite{BOP15, CL22, OO20, OOP19} for those are the most relevant to the current work, 
and see also \cite{FX19, FX21} with the approach using subcritical approximation. 
It is noteworthy that all of these works focus on the Euclidean space. 
To the best of our knowledge,
the result presented in this work is the first study of
global result for the energy-critical SNLS in the periodic setting.

\subsection{Main result}
Our main goal in this paper is to prove global well-posedness of the defocusing energy-critical SNLS \eqref{SNLS}. We say that $u$ is a solution to \eqref{SNLS} if it satisfies the following Duhamel formulation (or mild formulation):
\[
u(t)
= S(t)u_0
    -i\int_0^t S(t-t') (|u|^4u) (t')dt'
    -i\int_0^t S(t-t')\phi\xi(dt'),
\]

\noi
where $S(t)\coloneqq e^{it\Delta}$ denotes the linear Schr\"{o}dinger propagator.
The last term is known as the stochastic convolution and we denote it by
\begin{equation}\label{Psi}
    \Psi(t)\coloneqq -i\int_0^t S(t-t')\phi\xi(dt'),
\end{equation}

\noi
see \eqref{Psi'} below for a precise definition. The regularity of $\Psi(t)$ is dictated by the nature of $\phi$. In particular, if we assume that $\phi\in\HS(L^2(\T^3);H^s(\T^3))$ for appropriate values of $s \in \R$, guaranteeing that $\Psi\in C(\R_+; H^s(\T^3))$ almost surely, see Lemma \ref{Lem:sconv}. In Subsection \ref{Sec:s-conv}, we give a further discussion on the stochastic convolution.

The following is the main result of this paper.
\begin{theorem}
\label{Th}
Let $\phi\in\HS(L^2(\T^3);H^1(\T^3))$. 
Then, the defocusing energy-critical SNLS \eqref{SNLS} is globally well-posed
in $H^1(\T^3)$.
In particular, solutions are unique in the class
\[
\Psi + X^1 (\R_+),
\]
where
$X^1(\R_+)$
is defined in \eqref{XI} below.
\end{theorem}

Let us briefly go over our method of the proof. 
We first notice that we apply the first order expansion to 
a solution to  \eqref{SNLS}, which we denote as 
$ u =v+\Psi$.
Then, by using this solution,
we now rewrite the equation \eqref{SNLS} and 
$v$ satisfies
\begin{align}\label{SNLS2}
\begin{cases}
i \dt v + \Dl  v =|v|^4v+ e \\
v|_{t = 0} = u_0.
\end{cases}
\end{align} 

\noi
Here, $e= e(v,\Psi)$ is defined to be
\begin{align}
\label{pert}
\begin{split}
e &:=|v+\Psi|^{4} (v+\Psi)-|v|^4v.
\end{split}
\end{align} 

\noi
The equation \eqref{SNLS2} can be viewed as the energy-critical NLS \eqref{NLS1} with the perturbation $e$.

Our main approach to prove Theorem \ref{Th} is that we adapt the atomic spaces machinery to the stochastic setting.
Our  main difficulty comes from the local theory, particularly,
the stochastic convolution does not belong to a $V^2$-type space,
which is a space of bounded $2$-variation functions.
See \eqref{eq:norm_v} below.
Indeed,
L\'evy's modulus of continuity theorem
shows that
the Brownian motion has a finite $p$-variation
if and only if $p >2$.
Hence, the stochastic convolution belongs only to a $V^p$-type space for $p>2$.
See also \cite{CO1}.
We then notice that the proof of \cite[Proposition 3.5]{HTT11} cannot be directly applied to terms involving stochastic objects.
However, thanks to the spatial integrability of the stochastic convolution,
we can still obtain trilinear estimates that involve  stochastic terms (Proposition \ref{Prop:uff}).

As for the quintic-linear estimates,
we decompose the nonlinearity into three parts: 
the purely deterministic contribution, the purely stochastic contribution, 
and the mixed interaction terms involving both the deterministic component and the stochastic forcing term.
The machinery of atomic spaces is primarily employed to handle the purely deterministic contribution, 
following the well-established results in \cite{HTT11, IP12}. 
The purely stochastic part can be estimated directly using elementary inequalities; 
see the beginning of the proof of Proposition \ref{Prop:non est} for details.
The mixed interaction terms constitute the main novelty of this paper. 
These terms require a careful combination of atomic space techniques and stochastic calculus. 
We then apply careful case-by-case analysis with Proposition \ref{Prop:uff}.

In the global theory,
the regularity properties of $\Psi$ (Subsection \ref{Sec:s-conv} below) 
and the a priori bound on the energy will allow us to invoke 
the perturbation lemma from \cite{IP12} (Lemma \ref{Lem:Per} below) 
on $e$ iteratively over any given target time $T>0$
 to construct a solution $v$ to \eqref{SNLS2}. Hence, we finish our proof of  Theorem \ref{Th}.

\begin{remark}
\rm
Theorem \ref{Th} extends to $\T^4$ with a slightly modification on the trilinear and 
quinticlinear estimates;
see \cite{Y21} for the deterministic study and then an approach similar to that of this paper applies.

\end{remark}

The rest of the paper is organized as follows.
In Section \ref{SEC:pre}, 
we introduce some notations, 
recall atomic spaces machinery, 
and state  regularity properties of the stochastic convolution.
Then, we give a proof of local well-posedness of the perturbed NLS \eqref{SNLS2}
 in Section
\ref{SEC:Local},
we also present the key perturbation lemma. 
Next, we apply the perturbation lemma to show the global existence of solutions to the deterministic perturbed NLS.
Finally, Section
\ref{SEC:Th}
is devoted to proving Theorem \ref{Th}.

\section{Preliminary results and atomic spaces}
\label{SEC:pre}
In the following, we denote that $A\lesssim B$ if $A\leq CB$ for some constants $C > 0$. We also write $A\ll B$ if the implicit
constant is regarded as small.

Set
\[
\int_{\T^3} f(x) dx
:=
\frac 1{(2\pi)^3} \int_{[-\pi,\pi]^3} f(x) dx.
\]
We define the spatial Fourier coefficients
\begin{equation*}
    \widehat{f}(\zeta):= \int_{\T^{3}} e^{-i \zeta \cdot x}
f(x) dx
    \quad \text{for }\zeta \in \Z^{3}.
\end{equation*}
For periodic distributions $f$,
the Fourier series expansion holds:
\[
f(x) = \sum_{\zeta \in \Z^3} \ft f(\zeta) e^{i \zeta \cdot x}.
\]

Given $1 \le q,r\leq \infty$ and a time interval $I\subset\R$,
we consider the mixed Lebesgue space $L^q(I; L^r(\T^3))$ of space-time functions $f(t,x)$.
In the computation, we may even write $L^q (I; L^r(\T^3)) =L_t^q L_x^r (I \times \T^3)$
and $L^q([0,T]; L^r(\T^3)) = L_T^qL_x^r$ for $T>0$.

We fix a non-negative, even function $\psi\in C_0^\infty((-2,2))$ with $\psi(s)=1$ for $|s|\leq 1$,
in order to define a partition of unity:
for a dyadic number $N\geq 1$,
we set
\begin{equation*}
    \psi_N(\zeta)=\psi\Big(\frac{|\zeta|}{N}\Big)-\psi\Big(\frac{2|\zeta|}{N}\Big) \,
    \quad \text{for } N\geq 2,
\end{equation*}

\noi
and $  \psi_1(\zeta)=\psi(|\zeta|)$.
We define the frequency localization operators
$P_N$ as the Fourier multiplier with symbol $\psi_N$.
More generally, for a set $S\subset \Z^{3}$,
we define $P_S$ as the Fourier projection operator with symbol $\ind_S$, where $\ind_S$ denotes the characteristic function of $S$.

Let $s \in \R$.
We denote the Fourier multiplier with the symbol $\jb{\zeta}^s$ as $\jb{\nb}^s$.
For $1 \le q \le \infty$,
we define the Sobolev space $W^{s,q}(\T^3)$ as the space of all periodic distributions $f$ for which the norm
\[
\| f \|_{W^{s,q}}
:= \| \jb{\nb}^s f \|_{L^q}
< \infty.
\]
We write $H^s(\T^3) := W^{s,2}(\T^3)$.
The square function estimate (see \cite[Theorem 5 on page 104]{St70}, for example) yields that
\begin{equation}
\| f \|_{W^{s,q}}
\sim
\bigg\| \Big( \sum_{N \ge 1} N^{2s} |P_N f|^2 \Big)^{\frac 12} \bigg\|_{L^q}
\label{chrW}
\end{equation}
for $1<q<\infty$.

Given $1 <q,r<\infty$, $s \in \R$, and a time interval $I$,
we define
the space
$\l^2 L_t^q W_x^{s,r} (I \times \T^3)$
equipped with the norm
\[
\| f \|_{\l^2 L_t^q W_x^{s,r} (I \times \T^3)}
:=
\Big( \sum_{N \ge 1} N^{2s} \| P_N f\|_{L_t^q L_x^{r}(I \times \T^3)}^{2} \Big)^{\frac 12}
.
\]
Minkowski's integral inequality with \eqref{chrW} yields that
\begin{equation}
\| f \|_{L_t^q W_x^{s,r} (I \times \T^3)}
\les
\| f \|_{\l^2 L_t^q W_x^{s,r} (I \times \T^3)}
\label{sqLW}
\end{equation}
for $2 \le q,r < \infty$.

\subsection{Atomic spaces machinery}
Throughout this section let $H$ be a separable Hilbert space over $\mathbb C$.
We recall the well-known atomic spaces $U^{p}$
and $V^{p}$ machinery in this subsection.
This technique has made great progress in treating critical problems in recent years; 
for instance 
see those development (definition to the atomic spaces and related properties)
in \cite{HHK09, HTT11, HTT14, IP12, KT05}.

Let $\mathcal{Z}$ be the set of finite partitions
$-\infty<t_0<t_1< \dots <t_K\leq \infty$ of the real line. If
$t_K=\infty$, we use the convention that $v(t_K):=0$ for all functions
$v: \R \to H$.
Then, we give the same definition of $U^p$ and $V^p$ spaces as in \cite{HHK09}.

\begin{definition}\label{def:u}
\rm
Let $1\leq p <\infty$. For $\{t_k\}_{k=0}^K \in \mathcal{Z}$ and
$\{\phi_k\}_{k=0}^{K-1} \subset H$ with
$\sum_{k=0}^{K-1}\|\phi_k\|_{H}^p=1$.
We call the piecewise
defined function $a:\R \to H$ such that
\begin{equation*}
    a(t) =\sum_{k=1}^K\ind_{[t_{k-1},t_k)} (t) \phi_{k-1}
\end{equation*}

\noi
a $U^p$-atom,
 and then we define the atomic space $U^p(\R;H)$ of all
functions $u: \R \to H$ such that
\begin{equation*}
    u=\sum_{j=1}^\infty \lambda_j a_j \; 
 \quad 
 \text{  for  }\,
  U^p\text{-atoms } a_j \;
  \quad \text{ and } \quad
    \{\ld_j\}_{j \in \N} \in \l^1(\N; \C),
\end{equation*}

\noi
with the norm
\[
    \|u\|_{U^p(\R; H)} :=\inf \Big\{\sum_{j=1}^\infty |\lambda_j|
      :\; u=\sum_{j=1}^\infty \lambda_j a_j,
      \, \{ \lambda_j \}_{j \in \N} \in \l^1 (\N; \C) , \; a_j \, \text{is  $U^p$-atom}\Big\}.
\]
Here, $\l^1(\N; \C)$ denotes the space of all absolutely summable $\C$-valued sequences.
\end{definition}

\begin{definition}\label{def:v}
\rm
Let $1\leq p<\infty$.
We define $V^p(\R;H)$ as the space of all functions $v:\R\to H$ such that
\begin{equation}\label{eq:norm_v}
    \|v\|_{V^p(\R; H)}^p
    :=\sup_{\{t_k\}_{k=0}^K \in \mathcal{Z}} \sum_{k=1}^{K}
    \|v(t_{k})-v(t_{k-1})\|_{H}^p
\end{equation}
is finite\footnote{Notice that here we use the convention $v(\infty)=0$}.
Furthermore, let $V^p_{rc}(\R;H)$ denote the closed subspace of all right-continuous functions $v:\R\to H$ such that $\lim_{t\to -\infty}v(t)=0$, endowed with the same norm \eqref{eq:norm_v}.

\end{definition}

It is worth mentioning here that all the spaces defined in Definitions \ref{def:u} and \ref{def:v} are 
Banach spaces, and
embedded into   $L^\infty(\R;H)$.
Moreover, for  $1\leq p<q <\infty$,
we have $V^p_{rc}(\R;H) \hookrightarrow U^q(\R;H)$.
We refer to \cite{HHK09} for details.

Corresponding to the linear Schr\"odinger flow,
we define the space-time spaces as follows.

\begin{definition}\label{def:delta_norm}
\rm
For $s \in \R$, we let $U^p_\Delta H^s $ resp. $V^p_\Delta H^s$ 
be the spaces of all functions $u:\R\to H^s(\T^3)$ such that $t \mapsto e^{-it \Delta}u(t)$ is in $U^p(\R;H^s(\T^3))$ resp. $V^p(\R;H^s(\T^3))$, with norms
\[
    \| u\|_{U^p_\Delta H^s} = \| e^{-it \Delta} u\|_{U^p(\R;H^s(\T^3))},
    \qquad 
    \| u\|_{V^p_\Delta H^s} = \| e^{-it \Delta} u\|_{V^p(\R;H^s(\T^3))}.
\]
\end{definition}

Next, we also define spaces that only have a small variation of the above.

\begin{definition}\label{def:xt}
\rm
(i)
For $s \in \R$, we define $\ft X^{s}$ as the space of all functions $u: \R \to H^s(\T^3)$ 
such that, for every $\zeta \in \Z^3$, the map $t \mapsto e^{it|\zeta|^2} \ft{u}(t)(\zeta)$ is in $U^2(\R,\C)$, and for which the norm
\[
    \|u\|^2_{\ft X^{s}}:=\sum_{\zeta \in \Z^3} \jb{ \zeta } ^{2s}
    \|e^{it|\zeta|^2}\ft{u}(t)(\zeta)\|_{U^2_t(\R; \C)}^2<\infty.
\]

\noi
We define a Banach subspace of continuous functions in  $\ft X^{s}$
to be $X^s=\ft X^{s} \cap C(\R;H^s(\T^3))$.

\smallskip
(ii)
For $s \in \R$, we define $\ft Y^{s}$ as the space of all functions $u:\R \to H^s(\T^3)$ such that,
for every 
$\zeta\in \Z^3$, the map $t \mapsto e^{it|\zeta|^2}\ft u(t) (\zeta)$ is in $V^2_{rc}(\R,\C)$, and for which the norm
\[
    \|u\|^2_{\ft Y^{s}}
    :=\sum_{\zeta\in \Z^3}
    \jb{ \zeta } ^{2s} \|e^{it|\zeta|^2}\ft{u}(t)(\zeta)\|_{V^2_t (\R; \C)}^2 <\infty.
\]

\noi
Similarly, we define $Y^s=\ft Y^{s} \cap C(\R;H^s(\T^3))$.

\end{definition}

We note that functions in $\ft X^{s} $ and $\ft Y^{s} $
are not necessarily continuous, 
such functions are, 
 of course, in $L^\infty (\R; H^s (\T^3))$.
 In particular,
the following embeddings hold. See \cite{HTT11} for details.

\begin{proposition}\label{Prop:emb}
For $s \in \R$,
we have
\[
   U^2_\Delta H^s \hookrightarrow \ft X^s \hookrightarrow \ft Y^s
  \hookrightarrow V^2_\Delta H^s
  \hookrightarrow L^\infty (\R; H^s (\T^3)).
\]
Moreover, for  $1\leq p<q <\infty$,
we have $V^p_{\Dl} H^s \hookrightarrow U^q_{\Dl}H^s$.

\end{proposition}

As an immediate consequence of the above embedding results, 
we have the following result.

\begin{corollary}\label{est-Ys}
Let $\Z^3 = \bigcup_{k \in \N} C_k$ be a partition of $\Z^3$.
For $s \in \R$,
we have
\[
   \sum_{k \in \N} \| P_{C_k} u\|_{V^2_\Delta H^s}^2 \les   \| u\|^2_{Y^s}.
\]
\end{corollary}

For $s \in \R$ and $I\subset \R$,
we define the  restriction norms $X^s(I)$ in the following way:
\begin{equation}
\begin{aligned}
X^s(I)
&:= \{u\in C(I; H^s (\T^3)):
    \ u = v|_I \ \text{ for some } v \in X^s \},
\\
\| u \|_{X^s(I)}
&:=
\inf \{ \| v \|_{X^s} :
\ v \in X^s \text{ such that } u= v|_I \}
,
\end{aligned}
\label{XI}
\end{equation}
where $v|_I$ denotes the restriction of $v$ to $I$.
The space $Y^s(I)$ is defined in a similar way. 
We referee to \cite[Appendix A]{BOP15} to the additional
properties of the $X^s(I)$-spaces.

\subsection{Linear and Strichartz estimates}

The following linear estimate follows immediately from the atomic structure of $U^2$.

\begin{proposition}\label{Prop:linear}
  Let $s\geq 0$, $T>0$, and $\wt\phi \in H^s(\T^3)$. Then, for
  the linear solution $u(t) := e^{it\Delta} \wt\phi$ 
  for $t \in \R$, we have
  $u \in X^s([0,T))$ and
\[
    \|u\|_{X^s([0,T))}\leq \| \wt\phi\|_{H^s}.
\]
\end{proposition}

Define the norm of the inhomogeneous term on an interval $I= (a,b)$ by
\begin{equation}
\| f \|_{N^s(I)}
:=
\bigg\| \int_a^t S(t-t') f(t') dt' \bigg\|_{X^s(I)}
\label{NsI}
\end{equation}
for $s \in \R$.
We have the following linear estimate for the Duhamel term. 
See \cite[Proposition 2.11]{HTT11}

\begin{proposition}\label{Prop:inhom}
Let $s\geq 0$ and let $I$ be an interval.
For $f\in L^1(I;H^s(\T^3))$, we have
\[
\| f \|_{N^s(I)}
\leq
\sup
_{\substack{ v \in Y^{-s}(I) \\ \|v\|_{Y^{-s}(I)}=1}}
    \left|\int_I \int_{\T^3}f(t,x)\overline{v(t,x)}dxdt\right|.
\]
\end{proposition}

Next, we recall some estimates, which are similar to Bourgain's $L^p$ estimates of Strichartz type. Let $\mathcal{C}_N$ denote the collection of cubes $C\subset \Z^3$ of side-length $N\geq 1$ with arbitrary center and orientation. Then, we give the following estimates, and one can find the proofs in \cite[Corollary 3.2]{HTT11}.

\begin{corollary}
\label{Coro:str cubes}
Let $p>4$ and an interval $I$ with $|I| \le 1$.
For all dyadic numbers $N\geq 1$ and $C\in\mathcal{C}_N$, we have
\[
\|P_C u\|_{L^p(I\times\T^3)}
\les N^{\frac{3}{2}-\frac{5}{p}}\|P_C u\|_{U_\Delta^p L^2}.
\]
\end{corollary}

Let $s \in \R$ and $I$ be an interval with $|I| \le 1$.
As in \cite{IP12},
we define the norm:
\begin{align}
\label{weak-norm}
\| u \|_{\wt X^s(I)} :=
\sum_{p \in \{ p_0, p_1 \} }
\sup_{\substack{J \subset I \\ |J| \le 1}}
\Big(
\sum_{N \ge 1}
N^{5+ (s- \frac 32)p} \| P_N u(t) \|^p_{L^p_{t,x} (J \times \T^3) }
\Big)^{\frac 1p},
\end{align}

\noi
where $p_0$ and $p_1$ are given by
\begin{equation}
p_0 := 4 + \frac 1{10},
\quad
p_1 := 100.
\label{p0a}
\end{equation}
As a consequence of Corollary \ref{Coro:str cubes}, we 
note that the $\wt X^s$-norm is weaker than the $X^s$-norm:
\begin{align}
\label{eq_wenorm}
\| u \|_{\wt X^s(I)} \les \| u\|_{X^s(I)}.
\end{align}

\noi
Since $X^s(I) \hookrightarrow L^\infty(I; H^s(\T^3))$,
the size of the $X^s(I)$-norm is never small as we shrink the time interval.
On the other hand,
by taking $|I| \ll 1$, we can make the $\wt X^s(I)$-norm small.

We also define
\begin{align}
\label{M-norm}
\| u \|_{M(I)} :=
\|u\|_{\wt X^1(I)}^{\frac 12} 
\|u\|_{X^1(I)}^{\frac 12}.
\end{align}

\noi
We have the following trilinear estimate.

\begin{proposition}{\cite[Proposition 3.1]{IP12}}
\label{Prop:uuu}
There exists $\delta>0$ such that, for any dyadic numbers $N_1\geq N_2 \geq N_3\geq
1$ and any interval $I$ with $|I|\leq 1$, we have
\[
\Big\| \prod_{j=1}^3P_{N_j} u_j \Big\|_{L^2(I\times \T^3)}\les
\Big( \frac{N_3}{N_1}+\frac{1}{N_2} \Big)^\delta
\|P_{N_1}u_1\|_{Y^0(I)}
\|P_{N_2}u_2 \|_{M(I)}\|P_{N_3}u_3 \|_{M(I)}.
\]

\end{proposition}

The next proposition provides an essential estimate for dealing with terms involving stochastic objects. 
Thanks to the regularity properties of stochastic convolution (as in Lemma \ref{Lem:sconv}), 
we can always place stochastic objects in $L^p(I\times \T^3)$ for some $p$. 
With a rather direct proof, we obtain the following results.

\begin{proposition}\label{Prop:uff}
Let $N_1\geq N_2\geq N_3\geq 1$ be dyadic numbers and let $I$ be an interval with $|I| \le 1$.
Assume that functions $u_\l$ and $f_\l$ satisfy
\[
P_{N_\l} u_\l = u_\l, \quad
P_{N_\l} f_\l = f_\l
\]
for $\l=1,2,3$.
For any permutation $(i,j,k)$ of $(1,2,3)$ and $p>4$,
we have
\begin{align}
\label{eq:uff}
\|
u_i f_j f_k
\|_{L^2_{t,x}(I\times \T^3)}
&\lesssim
N_2^{\frac{3}{2}-\frac{5}{p}}
\| u_i\|_{Y^0 (I)}
\| f_j\|_{L_{t,x}^{\frac{4p}{p-2}}(I\times \T^3) }
\| f_k\|_{L_{t,x}^{\frac{4p}{p-2}}(I\times \T^3) },
\\
\label{eq:uuf}
\| f_i u_j u_k \|_{L_{t,x}^2(I\times\T^3)}
&\les
N_2^{3-\frac{10}{p}}
\| f_i\|_{L_{t,x}^{\frac{2p}{p-4}}(I\times\T^3)}
\| u_{j} \|_{Y^0(I)}
\| u_{k} \|_{Y^0(I)}
.
\end{align}
\end{proposition}

\begin{proof}
Since \eqref{eq:uuf} follows from a similar calculation,
we only prove \eqref{eq:uff}.
We  focus on the estimate
\begin{equation}
\|u_1 f_2 f_3\|_{L^2_{t,x} (I \times \T^3) }
\\
\lesssim
N_2^{\frac{3}{2}-\frac{5}{p}}
\|u_1\|_{Y^0(I)}
\|f_2\|_{L^q_{t,x} (I \times \T^3)}
\|f_3\|_{L^q_{t,x} (I \times \T^3)}
\label{eq:uff1}
\end{equation}
for $p>4$ and $q$ with $\frac{1}{p}+\frac{2}{q}=\frac{1}{2}$.
The other cases, such as
\begin{equation*}
    \| f_1 u_2 f_3\|_{L^2_{t,x} }
    \quad\quad
    \text{and}
    \quad\quad
    \| f_1 f_2 u_3\|_{L^2_{t,x}},
\end{equation*}
are similarly handled.

We use the same argument as in \cite[Proposition 3.5]{HTT11}.
Given a partition $\Z^3 = \bigcup_{k \in \N} C_k$ into cubes $C_k \in \mathcal{C}_{N_2}$ of size $N_2$,
the output
$(P_{C_k} u_1) f_2 f_3$
are almost orthogonal in $L^2(\T^3)$.
Hence, we obtain
\begin{equation}
\begin{aligned}
\|  u_1  f_2  f_3 \|^2_{L^2_{t,x} (I \times \T^3)}
&\les
\sum_{k \in \N}
\| (P_{C_k} u_1)  f_2 f_3\|^2_{L^2_{t,x} (I \times \T^3) }
\\
&\les
\sum_{k \in \N}
\|P_{C_k} u_1\|^2_{L^p_{t,x} (I \times \T^3)}
\| f_2\|^2_{L^q_{t,x} (I \times \T^3)} \| f_3\|^2_{L^q_{t,x} (I \times \T^3)}
\end{aligned}
\label{eq:uff2}
\end{equation}
for
$p>4$ and $q$ with $\frac{1}{p}+\frac{2}{q}=\frac{1}{2}$.
Corollary \ref{Coro:str cubes} and Proposition \ref{Prop:emb}
imply that
\[
\|P_{C_k} u_1\|_{L^p_{t,x} (I \times \T^3)}
\les
N_2^{\frac 32-\frac{5}{p}} \| \ind_I P_{C_k} u_1\|_{U^p_\Delta L^2}
\les
N_2^{\frac 32-\frac{5}{p}} \| \ind_I P_{C_k} u_1\|_{V^2_\Delta L^2}.
\]
By using Corollary \ref{est-Ys} and the square summability of the $Y^0$-norm, we have
\begin{equation}
\Big(
\sum_{k \in \N}
\|P_{C_k} u_1\|^2_{L^p_{t,x} (I \times \T^3)}
\Big)^{\frac 12}
\les
N_2^{\frac 32-\frac{5}{p}} 
\| u_1 \|_{Y^0(I)}.
\label{eq:uff4}
\end{equation}
By \eqref{eq:uff2} and \eqref{eq:uff4},
we obtain \eqref{eq:uff1}.
\end{proof}

\subsection{Regularity of stochastic convolutions}\label{Sec:s-conv}
In this subsection,
we give a precise definition of the stochastic convolution $\Psi$ (see \eqref{Psi}).
Moreover, we record its standard regularity properties.

Given two separable Hilbert spaces $H$ and $K$, 
we denote by $\HS(H; K)$ the space of Hilbert-Schmidt operators
$\phi$ from $H$ to $K$, endowed with the norm:
\begin{equation*}
    \|\phi\|_{\HS(H;K)}
    = \Big(\sum_{n\in\mathbb{Z}}\|\phi h_n\|_K^2 \Big)^{\frac{1}{2}},
\end{equation*}

\noi
where $\{h_n\}_{n \in \Z}$ is an orthonormal basis of $H$.
Note that the Hilbert-Schmidt norm does not depend on the choice of $\{h_n\}_{n\in \Z}$.

Let $(\Omega,\mathcal{F},\mathbb{P})$ be a probability space.
Let $W$ be the $L^2(\T^3)$-cylindrical Wiener process given by
\begin{equation*}
    W(t,x,\omega)
    \coloneqq
    \sum_{n\in\mathbb{Z}^3}
    \beta_n(t,\omega)e_n(x),
\end{equation*}

\noi
where $\{\beta_n\}_{n\in\mathbb{Z}^3}$ is a family of independent complex-valued Brownian motions
and $e_n(x)\coloneqq e^{i n\cdot x},$ $n\in\mathbb{Z}^3$. 
The space-time white noise $\xi$ is defined as the (distributional) time derivative of $W$, i.e., formally represented as ``$\xi dt=d W$''.
Here, the spatial regularity of $W$ is very rough.
More precisely, for each fixed $t\geq 0$, $W(t)$ belongs to $H^{-\frac{3}{2}-\eps}(\T^3)$ almost surely for any $\eps > 0$.

As mentioned earlier, we apply the operator $\phi \in \HS(L^2(\T^3); H^s(\T^3))$ for some $s \in \R$ and consider a smoothed-out version $\phi W$.
Then,
the stochastic convolution
$\Psi(t)$ is defined as follows
\begin{equation}\label{Psi'}
    \Psi(t)
    =-i\int_0^t S(t-t')\phi dW(t')
    = -i \sum_{n\in\mathbb{Z}^3}
   \phi( e_n)
    \int_0^t e^{-i(t-t')|n|^2} d\beta_n(t').
\end{equation}

Next, we state the regularity properties of the stochastic
convolution.

\begin{lemma}\label{Lem:sconv}
Let $T>0$ and $s \in \R$.
Suppose that 
$\phi\in \HS(L^2(\mathbb{T}^3); H^s(\T^3))$.

\noi
\textup{(i)} 
$\Psi\in C([0,T];H^s(\T^3))$ almost surely. 
Moreover, for any $1\leq p<\infty$, we have
\[
    \mathbb{E}
    \Big[\sup_{0\leq t\leq T}\|\Psi (t) \|^p_{H^s}\Big]
    \leq
    C\|\phi\|_{\HS(L^2;H^s)}^p,
\]
where $C=C(T,p)>0$.

\noi
\textup{(ii)}
For $1\leq q<\infty$ and $2\leq r<\infty$,
we have
$\Psi\in L^q([0,T];W^{s,r}(\T^3))$ almost surely.
Moreover,
for any $1\leq p<\infty$, it holds that
\begin{align}\label{eq:reg s-t}
    \mathbb{E}
    \big[
    \|\Psi\|_{L^q_T W_x^{s,r}}^p
    \big]
    \leq
    C\|\phi\|^p_{\HS(L^2;H^s)},
\end{align}
where $C=C(T,p,q,r)>0$.

\end{lemma}

The continuity of the additive noise for part \textup{(i)} 
can be  in \cite[Lemma 3.4]{CM18}; 
and for 
part \textup{(ii)} one can follow a similar proof as in
\cite[Lemma 2.8]{BLL24}. We omit the proof.
Replacing $L_T^q W_x^{s,r}$ in \eqref{eq:reg s-t} with $\l^2 L_T^q W_x^{s,r}$ yields the following lemma.

\begin{lemma}\label{Lem:sconv1}
Let $d\geq 1$, $T>0$, and $s \in \R$.
Suppose that $\phi\in \HS(L^2(\mathbb{T}^3); H^s(\T^3))$.
Then,
for any $1\leq p,q,r<\infty$, we have
\[
    \mathbb{E}
    \big[
   \| \Psi\|_{\l^2 L^q_T W_x^{s,r}}^p
    \big]
\leq
    C\|\phi\|^p_{\HS(L^2;H^s)},
\]
where $C=C(T,p,q,r)>0$.
\end{lemma}

\begin{proof}
For $p \ge 2$,
Minkowski's integral inequality
and
Lemma \ref{Lem:sconv}
yield that
\begin{align*}
\big\| \| \Psi\|_{\l^2 L^q_T W_x^{s,r}} \big\|_{L^p(\O)}
&\le
\big\| \| P_N \Psi \|_{L^p(\O; L^q_T W_x^{s,r})} \big\|_{\l^2_N}
\\
&\le
C
\big\| \| P_N \phi \|_{\HS(L^2; H^s)} \big\|_{\l^2_N}
\\
&=
C
\|\phi\|_{\HS(L^2;H^s)}.
\end{align*}
\end{proof}

\section{Energy-critical NLS with a perturbation}
\label{SEC:Local}

In this section, we first consider the following defocusing energy-critical NLS with a perturbation:
\begin{equation}\label{P NLS}
\begin{cases}
    i\partial_t v+\Delta v=
\NN(v+f)\\
    v|_{t=t_0}=v_0
\end{cases}
\end{equation}

\noi
for some
given deterministic function $f$ satisfying certain regularity conditions.
Here, 
we denote the nonlinearity by
\begin{align}
\label{P-non}
\NN(u):=
|u|^4 u.
\end{align}
After proving local well-posedness for \eqref{P NLS}, 
we transform \eqref{SNLS} into the form of \eqref{P NLS} (with $f =\Psi$) by first-order expansion. Combining this with the regularity properties of $\Psi$, we establish local well-posedness for \eqref{SNLS}.

To estimate the stochastic convolution,
we use the $Z(I)$-norm defined by
\begin{align}\label{norm:Z}
\|f\|_{Z(I)}
:=
\| f \|_{\l^2 L_t^{\frac{2p_0}{p_0-4}} W_x^{1, \frac{2p_0}{p_0-4}}(I\times\T^3)}
,
\end{align}
where $p_0$ is given by \eqref{p0a}.
Note that
\[
\frac{2p_0}{p_0-4}
= 82.
\]
We have the following nonlinear estimate.

\begin{proposition}\label{Prop:non est}
Let 
$I$ be an interval with $|I|\leq 1$.
Then,
for $v\in X^1(I)$ and $f \in Z(I)$,
we have
\begin{align}
\label{eq:non-est}
\begin{aligned}
\big \|
\NN (v+f)
\big \|_{N^1(I)}
&\les
\|v\|_{X^1(I)}
\|v\|_{M(I)}^4
+ |I|^\theta
\Big(
\|v\|_{X^1(I)}^2
\|v\|^2_{M(I)}
\|f\|_{Z(I)}
\\
&\qquad
+
\sum_{k=2}^5
\|v\|_{X^1(I)}^{5-k}
\|f\|_{Z(I)}^k
\Big),
\end{aligned}
\end{align}
where
$\theta >0$ is a constant
and
$N^1(I)$-, $M(I)$-, and $Z(I)$-norms are defined in \eqref{NsI}, \eqref{M-norm}, and \eqref{norm:Z}, respectively.
\end{proposition}

Before we proof the Proposition \ref{Prop:non est}, we make the following remark to 
comparing this estimate to its
deterministic counterpart.

\begin{remark}\rm
As we mentioned in the introduction,
the stochastic convolution belongs only
to a $V^p$-type space for $p>2$.
Hence, 
a separate analysis is required for those contributions involving stochastic components.
Since
the stochastic convolution has better spatial integrability,
we can use the $Z(I)$-norm to estimate the perturbation part $f$ instead of the $X^1(I)$-norm.
Moreover,
thanks to the improved integrability,
we also obtain the factor $|I|^\theta$ in the estimate of the parts containing $f$.
\end{remark}

\begin{proof}[Proof of Proposition \ref{Prop:non est}]

We first observe that 
when $f\equiv0$.
It follows from \cite[Lemma 3.2]{IP12} that
\begin{align*}
\big \|
\NN (v)
\big \|_{N^1(I)}
&\les
\|v\|_{X^1(I)} \|v\|^{4}_{M(I)}
.
\end{align*}
Moreover,
when $v\equiv0$,
Proposition \ref{Prop:inhom} with H\"older's inequality yields that
\begin{align*}
\big \|
\NN (f)
\big \|_{N^1(I)}
&\les
\| f^5 \|_{L_t^1 H_x^1 (I \times \T^3)}
\les
\| f \|_{L_t^5 H_x^1 (I \times \T^3)}
\| f \|_{L_t^5 L_x^\infty (I \times \T^3)}^4
\les
|I|^{1- 5\frac{p_0-4}{2p_0}}
\|f\|_{Z(I)}^5.
\end{align*}
Here,
we used the bound
\[
\| f \|_{L_t^5 H_x^1 (I \times \T^3)}
+
\| f \|_{L_t^5 L_x^\infty (I \times \T^3)}
\les
|I|^{\frac 15-\frac{p_0-4}{2p_0}}
\|f\|_{Z(I)},
\]
which follows from H\"older's and Sobolev's inequalities, \eqref{sqLW}, and \eqref{norm:Z}.

Next, we focus on the cross terms, defined as those containing at least one $v$ and one $f$ in the nonlinearity.
By Proposition \ref{Prop:inhom}, we have
\[
    \big\| \NN(v+f) \|_{N^1(I)}
    \lesssim
    \sup_{\|g\|_{Y^{-1} (I)}=1}
    \bigg|  \iint_{I\times\T^3} \mathcal{N}(v,f)(t,x)\overline{g(t,x)}dxdt  \bigg|.
\]
To obtain the factor $|I|^\theta$ for some $\theta>0$,
we use the norm
\[
\|f\|_{\wt Z(I)}
:=
\| f \|_{\l^2 L_t^{\frac{2q_0}{q_0-4}} W_x^{1, \frac{2q_0}{q_0-4}}(I\times\T^3)}
,
\]
where $q_0$ is given by
\begin{equation}
q_0 := 4 + \frac 12.
\label{p0aa}
\end{equation}
Note that
\[
\frac{2q_0}{q_0-4}
= 18
<
\frac{2p_0}{p_0-4}
.
\]
We need to prove the multilinear estimate
\begin{align}
\label{eq:non-est3}
\begin{aligned}
    \bigg|  \iint_{I\times\T^3}
\big(
\NN(v+f)
- \NN(v)
- \NN(f)
\big)
\cj{g}
dxdt  \bigg|
&\les 
\|g\|_{Y^{-1} (I)}
\Big(
\|v\|_{X^1(I)}^2
\|v\|_{M(I)}^2
\|f\|_{\wt Z(I)}
\\
&
\quad
+
\sum_{k=2}^4
\|v\|_{X^1(I)}^{5-k}
\|f\|_{\wt Z(I)}^k
\Big)
.
\end{aligned}
\end{align}

Set
$w_0 = g$.
Let $w_j$ denote either $v$ or $f$ for $j=1, \dots, 5$.
We dyadically decompose 
\[
w_j =\sum_{N_j\geq1} P_{N_j} w_j
\]
for $j=0, \dots, 5$.
We omit the conjugate signs as it does not contribute to our analysis.
From \eqref{eq:non-est3} and \eqref{P-non},
we estimate the following:
\begin{equation}
\begin{aligned}
\bigg| \int_{I\times \T^3} \prod_{j=0}^5  w_j   dxdt  \bigg|
&\le
\sum_{N_0, \dots, N_5 \ge 1}
\bigg| \int_{I\times \T^3} \prod_{j=0}^5  P_{N_j} w_j  dxdt  \bigg|
\\
&=:
\sum_{\Nb}
S_{\Nb},
\end{aligned}
\label{Sum0}
\end{equation}

\noi
where $\Nb := (N_0,\dots, N_5)$.
Note that $S_\Nb$ becomes zero
unless the largest two frequencies are comparable.

\medskip
\noi
{\bf {Case~A: $vvvvf$ case}}. 
\smallskip

Without loss of generality,  we may
assume $w_j=v$ for $j=1,\dots,4$ and $w_5=f$ 
such that
$N_1\geq N_2\geq N_3\geq N_4\geq 1$ and $N_5\geq 1$.
We consider $\Nb$ satisfying the following two cases:
\[
N_0 \les N_1
 \quad 
 \quad
 {\rm or}
 \quad 
 \quad
 N_0\sim N_5\gg N_1.
 \]

\medskip
\noi
{\bf Subcase A$_1$:
$N_0 \les N_1 $}.  
\smallskip

Combining Propositions \ref{Prop:emb}, \ref{Prop:uuu} and \ref{Prop:uff} with \eqref{Sum0}, we obtain \begin{align*}
S_{\Nb}
&\le
\|P_{N_0}g P_{N_2} v P_{N_5}f \|_{L_{t,x}^2}
\|P_{N_1}v P_{N_3}v P_{N_4}v\|_{L_{t,x}^2}
\\
&\les
\med(N_0,N_2,N_5)^{3-\frac{10}{q_0}}
\|P_{N_0}g\|_{Y^0(I)}
\|P_{N_2}v\|_{Y^0(I)}
\|P_{N_5}f\|_{L_{t,x}^{\frac{2q_0}{q_0-4}}(I \times \T^3)}
\\
&\qquad
\times
\Big(\frac{N_4}{N_1}+\frac{1}{N_3}\Big)^\delta
\|P_{N_1}v\|_{Y^0(I)}
\|P_{N_3}v\|_{M(I)}
\|P_{N_4}v\|_{M(I)}
\\
&\les
N_0^2 N_1^{-1}
\min(N_0,N_2,N_5)^{-1}
\max(N_0,N_2,N_5)^{-1}
\med(N_0,N_2,N_5)^{2-\frac{10}{q_0}}
\\
&\qquad
\times
\|P_{N_0} g \|_{Y^{-1}(I)}
\| P_{N_1} v \|_{X^1(I)}
\| P_{N_2} v \|_{X^1(I)}
\| v \|_{M(I)}^2
\| P_{N_5} f \|_{\wt Z(I)}
.
\end{align*}
Note that $N_0 \les N_1$ and $S_{\Nb} \neq 0$ are equivalent to
the following three cases:
\[
A_{1,1}:
\
N_0 \sim N_1 \ges N_5, \quad
A_{1,2}:
\
N_1 \sim N_5 \ges N_0, \quad
A_{1,3}:
\
N_1 \sim N_2 \ges N_0, N_5.
\]
Since
\eqref{p0aa} yields
$2-\frac{10}{q_0}<0$,
we have
\[
N_0^2 N_1^{-1}
\max(N_0,N_2,N_5)^{-1}
\med(N_0,N_2,N_5)^{2-\frac{10}{q_0}}
\les
\begin{cases}
\max (N_2,N_5)^{2-\frac{10}{q_0}},
& \text{if } A_{1,1}, \\
\max (N_0,N_2)^{2-\frac{10}{q_0}},
& \text{if } A_{1,2}, \\
N_1^{2-\frac{10}{q_0}},
& \text{if } A_{1,3}.
\end{cases}
\]
We therefore obtain
\[
\sum_{N_0 \les N_1} S_\Nb
\les
\| g \|_{Y^{-1}(I)}
\| v \|_{X^1(I)}^2
\| v \|_{M(I)}^2
\| f \|_{\wt Z(I)},
\]
which shows  \eqref{eq:non-est3}.

\medskip
\noi
{\bf Subcase A$_2$: $N_0\sim N_5\gg N_1$}. 
\smallskip

Propositions \ref{Prop:uuu} and \ref{Prop:uff} with \eqref{Sum0} and $N_0\sim N_5\gg N_1$ yield that
\begin{align*}
S_{\Nb}
&\le
\|P_{N_0} g P_{N_3}v P_{N_4}v \|_{L_{t,x}^2}
\| P_{N_1}v P_{N_2}v P_{N_5}f\|_{L_{t,x}^2}
\\
&\les
\Big(\frac{N_3}{N_0}+\frac{1}{N_2}\Big)^\delta
\|P_{N_0} g \|_{Y^0(I)}
\|P_{N_3}v\|_{M(I)}
\|P_{N_4}v\|_{M(I)}
\\
&\qquad
\times
N_1^{3-\frac{10}{q_0}}
\|P_{N_1}v\|_{Y^0(I)}
\|P_{N_2}v\|_{Y^0(I)}
\|P_{N_5}f\|_{L_{t,x}^{\frac{2q_0}{q_0-4}}(I \times \T^3)}
\\
&\les
N_1^{2-\frac{10}{q_0}}
\|P_{N_0} g\|_{Y^{-1}(I)}
\| v \|_{X^1(I)}^2
\| v \|_{M(I)}^2
\| P_{N_5} f \|_{\wt Z(I)}
.
\end{align*}
By $2-\frac{10}{q_0}<0$,
we have \eqref{eq:non-est3}.

\medskip
\noi
{\bf Case B: $vvvff$ case}. 
\smallskip

Without loss of generality, we may
assume $w_j=v$
 for $j=1,2,3$
and $w_i=f$ for $i=4,5$
such that $N_1\geq N_2\geq N_3\geq 1$ and $N_4\geq N_5\geq 1$. 
Proposition \ref{Prop:uff} with  \eqref{Sum0} yields that
\begin{align*}
S_{\Nb}
&\le
\|P_{N_0}g P_{N_2}v P_{N_5}f\|_{L_{t,x}^2}
\|P_{N_1}v P_{N_3}v P_{N_4}f\|_{L_{t,x}^2}
\\
&\les
\med(N_0,N_2,N_5)^{3-\frac{10}{q_0}}
\|P_{N_0}g\|_{Y^0(I)}
\|P_{N_2}v\|_{Y^0(I)}
\|P_{N_5}f\|_{L_{t,x}^{\frac{2q_0}{q_0-4}}(I \times \T^3)}
\\
&\qquad
\times
\med(N_1,N_3,N_4)^{3-\frac{10}{q_0}}
\|P_{N_1}v\|_{Y^0(I)}
\|P_{N_3}v\|_{Y^0(I)}
\|P_{N_4}f\|_{L_{t,x}^{\frac{2q_0}{q_0-4}}(I \times \T^3)}
\\
&\les
N_0^2
\min(N_0,N_2,N_5)^{-1}
\max(N_0,N_2,N_5)^{-1}
\med(N_0,N_2,N_5)^{2-\frac{10}{q_0}}
\\
&\qquad
\times
\min(N_1,N_3,N_4)^{-1}
\max(N_1,N_3,N_4)^{-1}
\med(N_1,N_3,N_4)^{2-\frac{10}{q_0}}
\\
&\qquad
\times
\|P_{N_0} g \|_{Y^{-1}(I)}
\bigg(
\prod_{j=1}^2 \| P_{N_j} v \|_{X^1(I)}
\| P_{N_{j+3}} f_{j+3} \|_{\wt Z(I)}
\bigg)
\| P_{N_3} v \|_{X^1(I)}
.
\end{align*}
We need to consider the following five cases:
\[
\begin{aligned}
&N_0 \sim N_1 \ges N_4, \quad
N_0 \sim N_4 \ges N_1, \quad
N_1 \sim N_4 \ges N_0, \\
&N_1 \sim N_2 \ges N_0, N_4, \
\text{ or } \
N_4 \sim N_5 \ges N_0, N_1.
\end{aligned}
\]
In all the cases,
we have
\[
N_0 \les
\max (N_1,N_4)
.
\]
In particular,
\[
N_0^2
\max(N_0,N_2,N_5)^{-1}
\max(N_1,N_3,N_4)^{-1}
\les 1.
\]
Thanks to the presence of
$\med(N_0,N_2,N_5)^{2-\frac{10}{q_0}} \med(N_1,N_3,N_4)^{2-\frac{10}{q_0}}$,
we obtain
\[
\sum_{\Nb} S_\Nb
\les
\|g\|_{Y^{-1}(I)}
\| v \|_{X^1(I)}^3
\| f \|_{\wt Z(I)}^2,
\]
which shows  \eqref{eq:non-est3}.

\medskip
\noi
{\bf Case C: $vvfff$ case}. 
\smallskip

This case is similarly handled as in Case B.
Without loss of generality,  we may
assume $w_j=v_j$ for $j=1,2$
and $w_i=f_i$ for $i=3,4,5$
such that 
$N_1\geq N_2\geq 1$ 
and $N_3\geq N_4\geq N_5\geq 1$.
Proposition \ref{Prop:uff} with \eqref{Sum0} yields that
\begin{align*}
S_{\Nb}
&\le
\|P_{N_0} g P_{N_2}v P_{N_5}f\|_{L_{t,x}^2}
\|P_{N_1}v P_{N_3}f P_{N_4}f\|_{L_{t,x}^2}
\\
&\les
\med(N_0,N_2,N_5)^{3-\frac{10}{q_0}}
\|P_{N_0}g\|_{Y^0(I)}
\|P_{N_2}v\|_{Y^0(I)}
\|P_{N_5}f\|_{L_{t,x}^{\frac{2q_0}{q_0-4}}(I \times \T^3)}
\\
&\qquad
\times
\med(N_1,N_3,N_4)^{\frac 32-\frac{5}{q_0}}
\|P_{N_1}v\|_{Y^0(I)}
\|P_{N_3}f\|_{L_{t,x}^{\frac{4q_0}{q_0-2}} (I \times \T^3)}
\|P_{N_4}f\|_{L_{t,x}^{\frac{4q_0}{q_0-2}} (I \times \T^3)}
\\
&\les
N_0^2
\min(N_0,N_2,N_5)^{-1}
\max(N_0,N_2,N_5)^{-1}
\med(N_0,N_2,N_5)^{2-\frac{10}{q_0}}
\\
&\qquad
\times
\min(N_1,N_3,N_4)^{-1}
\max(N_1,N_3,N_4)^{-1}
\med(N_1,N_3,N_4)^{\frac 12 -\frac{5}{q_0}}
\\
&\qquad
\times
\|P_{N_0}g\|_{Y^{-1}(I)}
\bigg(
\prod_{j=1}^2 \| P_{N_j} v \|_{X^1(I)}
\| P_{N_{j+2}} f \|_{\wt Z(I)}
\bigg)
\| P_{N_5} f \|_{\wt Z(I)}
.
\end{align*}
We need to consider the following five cases:
\[
\begin{aligned}
&N_0 \sim N_1 \ges N_3, \quad
N_0 \sim N_3 \ges N_1, \quad
N_1 \sim N_3 \ges N_0, \\
&N_1 \sim N_2 \ges N_0, N_3, \
\text{ or } \
N_3 \sim N_4 \ges N_0, N_1.
\end{aligned}
\]
In all the cases,
we have
\[
N_0 \les
\max (N_1,N_3).
\]
In particular,
\[
N_0^2
\max(N_0,N_2,N_5)^{-1}
\max(N_1,N_3,N_4)^{-1}
\les 1.
\]
Thanks to the presence of
$\med(N_0,N_2,N_5)^{2-\frac{10}{q_0}} \med(N_1,N_3,N_4)^{\frac 12-\frac{5}{q_0}}$,
we obtain
\[
\sum_{\Nb} S_\Nb
\les
\| g \|_{Y^{-1}(I)}
\| v \|_{X^1(I)}^2
\| f \|_{\wt Z(I)}^3,
\]
which shows  \eqref{eq:non-est3}.

\medskip
\noi
{\bf Case D: $vffff$ case}. 
\smallskip

This case also is similarly handled as in Case B.
We may
assume $w_1=v_1$,  
and $w_i=f_i$ for $i=2,3,4,5$
such that 
$N_1\geq 1$ 
and $N_2\geq N_3\geq N_4\geq N_5\geq 1$.
Proposition \ref{Prop:uff} with \eqref{Sum0} yields that
\begin{align*}
S_{\Nb}
&\le
\|P_{N_0}g P_{N_3}f P_{N_5}f\|_{L_{t,x}^2}
\|P_{N_1}v P_{N_2}f P_{N_4}f\|_{L_{t,x}^2}
\\
&\les
\med(N_0,N_3,N_5)^{\frac 32-\frac{5}{q_0}}
\|P_{N_0}g\|_{Y^0(I)}
\|P_{N_3}f\|_{L_{t,x}^{\frac{4q_0}{q_0-2}} (I \times \T^3)}
\|P_{N_5}f\|_{L_{t,x}^{\frac{4q_0}{q_0-2}} (I \times \T^3)}
\\
&\qquad
\times
\med(N_1,N_2,N_4)^{\frac 32-\frac{5}{q_0}}
\|P_{N_1}v\|_{Y^0(I)}
\|P_{N_2}f\|_{L_{t,x}^{\frac{4q_0}{q_0-2}} (I \times \T^3)}
\|P_{N_4}f\|_{L_{t,x}^{\frac{4q_0}{q_0-2}} (I \times \T^3)}
\\
&\les
N_0^2
\min(N_0,N_3,N_5)^{-1}
\max(N_0,N_3,N_5)^{-1}
\med(N_0,N_3,N_5)^{\frac 12-\frac{5}{q_0}}
\\
&\qquad
\times
\min(N_1,N_2,N_4)^{-1}
\max(N_1,N_2,N_4)^{-1}
\med(N_1,N_2,N_4)^{\frac 12 -\frac{5}{q_0}}
\\
&\qquad
\times
\|P_{N_0}g\|_{Y^{-1}(I)}
\| P_{N_1} v \|_{X^1(I)}
\bigg(
\prod_{j=2}^5
\| P_{N_j} f \|_{\wt Z(I)}
\bigg)
.
\end{align*}
We need to consider the following five cases:
\[
\begin{aligned}
&N_0 \sim N_1 \ges N_2, \quad
N_0 \sim N_2 \ges N_1, \quad
N_1 \sim N_2 \ges N_0, \
\text{ or } \
N_2 \sim N_3 \ges N_0, N_1.
\end{aligned}
\]
In all the cases,
we have
\[
N_0 \les
\max (N_1,N_2).
\]
In particular,
\[
N_0^2
\max(N_0,N_3,N_5)^{-1}
\max(N_1,N_2,N_4)^{-1}
\les 1.
\]
Thanks to the presence of
$\med(N_0,N_3,N_5)^{\frac 12 -\frac{5}{q_0}} \med(N_1,N_2,N_4)^{\frac 12-\frac{5}{q_0}}$,
we obtain
\[
\sum_{\Nb} S_\Nb
\les
\| g \|_{Y^{-1}(I)}
\| v \|_{X^1(I)}
\| f \|_{\wt Z(I)}^4,
\]
which shows  \eqref{eq:non-est3}.
\end{proof}

\subsection{Local well-posedness of the perturbed NLS}

We are now ready to present the local well-posedness of the perturbed NLS.
A similar proof can be found in \cite[Proposition 6.3]{BOP15}.
See also \cite[Proposition 3.3]{IP12}.

\begin{proposition}\label{Prop:loc}
Let 
$I =[t_0,t_1]\subset\R$ with $|I| \le 1$.
Suppose that 
\begin{align*}
\| v_0\|_{H^1(\T^3) }  \leq R \quad \text{and} \quad
\| f\|_{L^\infty_t H^1_x  (I\times \T^3)}\leq M
\end{align*}

\noi
for some $R,M>0$.
Then, there exists some small $\eta_0 = \eta_0(R, M) > 0$
such that if
\begin{equation*}
\max \big( \|S(t-t_0)v_0\|_{\wt X^1(I)},    \|f\|_{Z(I)} \big) \leq \eta
\end{equation*}

\noi
for some $\eta\leq  \eta_0$,
then
 there exists a unique solution
 $v\in X^1(I)$ 
to \eqref{P NLS}. Moreover, we have
\begin{equation}
\label{eq_cnt0}
    \|v-S(t-t_0)v_0\|_{X^1(I)}
    \les
    \eta,
\end{equation}
where the implicit constant is independent of $R$ and $M$.

%

\end{proposition}

\begin{proof}
Set
\[
 B_{\wt R, \eta}:=\{
 v\in  X^1(I)
 \; : \:
 \| v \|_{X^1(I)} \leq 2 \wt R ,\
 \| v\|_{\wt X^1 (I)} \leq 2 \eta\}
\]

 \noi
 where 
 $\wt R := \max(R,M,1)$.
It follows from \eqref{eq_wenorm} that $B_{\wt R, \eta}$ is a closed subset of $X^1(I)$ with respect to the metric
\[
d(u,v) := \| u-v \|_{X^1(I)}.
\]

We define the solution map $\Gamma$ of \eqref{P NLS} by
\begin{equation*}
    \Gamma(v)(t) \coloneqq
    S(t-t_0)v_0
    -i \int_0^t S(t-t') \NN(v+f) (t') dt'.
\end{equation*}
We show that the map $\Gamma$
 is contraction on $(B_{\wt R, \eta}, d)$.

We choose $\eta_0\ll \wt R^{-3}$.
 In particular, 
 we have $\eta_0 \ll \wt R^{-1} \leq 1$.
 Fix $\eta \leq \eta_0$ in the following.
By  Propositions \ref{Prop:linear} and \ref{Prop:non est},
\eqref{M-norm},
and Young's inequality,
we have
\begin{align}
\label{eq_cnt1}
\begin{aligned}
    \|\Gamma(v)\|_{X^1(I)}
    &\leq 
    \|S(t-t_0)v_0\|_{X^1(I)}
    +\| \Gamma(v) - S(t-t_0)v_0 \|_{X^1(I)}    \\
&\leq
\|v_0\|_{H^1}
+ C \big(
\wt R^3 \eta^2
+ \eta^5
\big)
\\
&
<2 \wt R
\end{aligned}
\end{align}
for 
$v \in  B_{\wt R, \eta}$.
The difference
is similarly handed:
\begin{equation*}
\begin{split}
    \|\Gamma(v_1)-\Gamma(v_2)\|_{X^1(I)}
    &=
    \big\| \NN(v_1+f) - \NN(v_2+f) \big\|_{N^1(I)}\\
    &\leq
    C\|v_1-v_2\|_{X^1(I)} \sum_{j=1}^2
\big(
\|v_1\|_{X^1 (I)}^2 \|v_j\|_{M(I)}^2
+ \|v_1\|_{X^1 (I)}^2 \| f\|_{Z(I)}^2
\\
&\qquad
+ \|v_1\|_{X^1 (I)} \|f\|_{Z(I)}^3
+\|f\|_{Z(I)}^4
    \big)\\
&\leq
C \wt R^3 \eta \|v_1-v_2\|_{X^1(I)} < \frac12 \|v_1-v_2\|_{X^1(I)}
\end{split}
\end{equation*}

\noi
for 
$v_1, v_2 \in  B_{\wt R, \eta}$.
Moreover, with \eqref{eq_wenorm} and \eqref{eq_cnt1},
we have
\begin{equation*}
\begin{split}
\|\Gamma(v)\|_{\wt X^1 (I)}
&\leq
\|S(t-t_0)v_0\|_{\wt X^1(I)}
+ C\| \Gamma(v) - S(t-t_0)v_0 \|_{X^1(I)}
\\
&<2\eta
\end{split}
\end{equation*}

\noi
for
$v\in  B_{\wt R, \eta}$.
Hence, $\G$ is a contraction on $(B_{\wt R, \eta}, d)$.
 The estimate \eqref{eq_cnt0} follows from 
 the choice of $\eta_0$
 and \eqref{eq_cnt1}.
\end{proof}

Next,
we will recall the stability theory of the defocusing energy-critical NLS. 
First, we recall the following global space-time bound on the solution to the defocusing energy-critical NLS on $\R\times\T^3$;
see \cite[Theorem 1.1]{IP12}.



\begin{lemma}\label{st-bound}
Let 
$w_0\in H^1(\T^3)$,
$t_0 \in \R$,
and
$w$ be the solution to the defocusing energy-critical NLS
\begin{equation}\label{Eq:w NLS}
\begin{cases}
    i\pa_t w+\Delta w=|w|^{4}w \\
  w|_{t=t_0}=w_0.
\end{cases}
\end{equation}

\noi
Then, there exists a unique global solution $w\in X^1(\R)$ such that 
the following space-time bound holds:
\begin{equation*}
    \|w\|_{X^1(\R)}\leq C(\|w_0\|_{H^1(\T^3)}).
\end{equation*}
\end{lemma}

We also recall the following perturbation lemma from 
\cite[Proposition 3.4]{IP12}.

\begin{lemma}\label{Lem:Per}
Let $v_0\in H^1(\T^3)$, $t_0 \in \R$, and $I$ is an open bounded interval containing  $t_0$. 
Assume $v$ is a solution to the following perturbed equation 
\begin{equation*}
\begin{cases}
    i\pa_t v+\Delta v=|v|^4v+e\\
    v|_{t=t_0}=v_0
\end{cases}
\end{equation*}

\noi
on $I\times\T^3$,
where $e$ is an given function.
Also, assume there exists some  $1\leq  R<\infty$ such that
 $v$ satisfies 
\begin{equation*}
    \|v\|_{\wt X^1(I)}
    +\|v\|_{L_t^\infty H_x^1(I\times\T^3)}\leq R.
\end{equation*}

\noi
Let $w_0\in H^1(\T^3)$ be the initial data of \eqref{Eq:w NLS} such that the smallness condition
\begin{align}
\label{eq_smallness}
\|v_0-w_0\|_{H^1(\T^3)}
+\|e \|_{N^1(I)}
\leq \eps
\end{align}

\noi
holds for some $0<\eps<\eps_0$, 
where $\eps_0$ is a small constant 
$0<\eps_0=\eps_0(R)<1$.
Then, there exists a solution $w\in X^1(I)$ to the 
defocusing energy-critical NLS \eqref{Eq:w NLS} and
\begin{align*}
\|v\|_{X^1(I)}+\|w\|_{X^1(I)}\leq C(R),
\quad
\quad
\|v-w\|_{X^1(I)}\leq C(R)\eps,
\end{align*}

\noi
where $C(R)$ is a nondecreasing function of $R$.
\end{lemma}


\section{Proof of global well-posedness }
\label{SEC:Th}
In this section, we prove the global well-posedness of the defocusing energy-critical SNLS \eqref{SNLS}. 
Its proof is based on 
an iterative  argument provided 
the almost sure energy bound (Proposition \ref{Prop:as E}),
the global-in-time space-time bound on the solution (Lemma \ref{st-bound}),
and the stability theory (Lemma \ref{Lem:Per}).

Proposition \ref{Prop:loc} with
Lemmas \ref{Lem:sconv} and
\ref{Lem:sconv1}
yields
the local well-posedness of \eqref{SNLS}.

\begin{corollary}
\label{COR_LWP}
Let $\phi\in\HS(L^2(\T^3);H^1(\T^3))$. 
 Then, given any $u_0 \in H^1(\T^3)$, there exists an almost surely positive stopping time $T = T_\o(u_0)$
 and a unique local-in-time solution
$u \in C([0, T]; H^1(\T^3))$ to the energy-critical SNLS \eqref{SNLS}. 
In particular, solutions are unique in the class
 \[
 \Psi + X^1 ([0,T]).
 \]

\end{corollary}

Then, next proposition provides a priori control on the energy $E(u(t))$ defined by
\begin{equation}\label{energy}
   E(u(t)) := \frac 12 \int_{\T^3} |\nb u(t, x)|^2 dx
+ \frac{1}{6} \int_{\T^3} |u(t, x)|^6 dx.
\end{equation}
The a priori bound on the energy follows from Ito's lemma and the Burkholder-Davis-Gundy inequality. 
Note that one requires a proper  justification on  the application of Ito's lemma, 
which needs to go through a certain approximation argument (with  local well-posedness theory; Corollary \ref{COR_LWP}).
See \cite[Proposition 5.1]{CM18} for a similar proof on $\T^d$,
 and \cite{BD03, CL22, CLO21, OO20} on $\R^d$.

\begin{proposition}\label{Prop:as E}
Let $\phi\in\HS(L^2(\T^3);H^1(\T^3))$ 
and $u_0\in H^1(\T^3)$. Let $u$
be the solution to the defocusing energy-critical SNLS \eqref{SNLS} and let $T^\ast=T^\ast(\omega,u_0)$ be the forward maximal time of existence. Then, given any $T_0>0$, there exists $C=C(\|u_0\|_{H^1},\|\phi\|_{\HS(L^2;H^1)},T_0)>0$ such that for any stopping time $T$ with $0<T<\min(T^\ast,T_0)$ almost surely, we have
\begin{equation*}
    \mathbb{E}
    \Big[\sup_{0\leq t\leq T}
    \big( \|u(t) \|_{L^2}^2 + E(u(t)) \big)
    \Big]
    \leq C.
\end{equation*}
\end{proposition}

Finally, we have everything we need to
prove Theorem \ref{Th}.
While the argument is the same as the proof of Proposition 7.2 in \cite{BOP15},
we give a proof here for completeness.

\begin{proof}[Proof of Theorem \ref{Th}]
Let $T>0$.
It follows from Lemmas \ref{Lem:sconv} and \ref{Lem:sconv1} with \eqref{norm:Z} that
\begin{equation}
\|\Psi\|_{L^\infty_T H^1_x}
+ \|\Psi\|_{Z([0,T])}
\le 
M(\o, T, \|\phi\|_{\HS(L^2;H^1)}) =:M.
\label{PsiM}
\end{equation}
Suppose that the solution $u$ to \eqref{SNLS} exists on $[0,T]$.
Then,
it follows from \eqref{energy},
\eqref{PsiM},
and Proposition \ref{Prop:as E}
that
\begin{align}
\label{eq_asbd}
\begin{aligned}
\|v\|_{L^\infty_T H^1_x }
&\le  \|u\|_{L^\infty_T H^1_x } + \|\Psi\|_{L^\infty_T H^1_x }
\le \sup_{0\le t\le T}
\big( \| u(t) \|_{L^2} + E(u(t))^\frac12 \big) + \|\Psi\|_{L^\infty_T H^1_x }\\
&\leq C(\omega, T, \| u_0 \|_{H^1}, \|\phi\|_{\HS(L^2;H^1)})
=:R .
\end{aligned}
\end{align}

We pick any $t_0\in (0,T)$ 
and suppose that the solution $v$ to \eqref{SNLS2} 
has already been constructed on $[0,t_0]$. 
We are going to show 
the existence of a unique solution $v$ 
to \eqref{SNLS2} on $[t_0,t_0+\tau]\cap[0,T]$ with $\tau>0$ independent of $v(t_0)$. 
Then, we can iterate the argument, and the
solution $v$ to \eqref{SNLS2} 
can be constructed on $[0,T]$. 
In this way, 
we prove Theorem \ref{Th}.

Let $w$ be the global solution to
the deterministic defocusing energy-critical NLS \eqref{Eq:w NLS}
with $w_0=v(t_0)$. 
By \eqref{eq_asbd}, we have $\| w(t_0)\|_{H^1}\leq R$.
From \eqref{eq_wenorm} and Lemma \ref{st-bound}, 
we obtain
\begin{equation}
    \|w\|_{\wt X^1([t_0,T])}\les
    \|w\|_{X^1([t_0,T])}\leq C(R).
\label{bdWI1}
\end{equation}

\noi
Note that the critical $X^1$-norm is never small when we shrink the time interval.
However, the weaker critical $\wt X^1$-norm (see \eqref{weak-norm})
shrinks when the time interval gets small.
Hence,
by letting $0<\eta\ll 1$ small to be chosen later,
we can divide the interval 
$[t_0,T]$ into $J=J(R,\eta)$ subintervals $I_j=[t_j,t_{j+1}]$
such that
\begin{equation}
\|w\|_{\wt X^1(I_j)}
\le \eta
\label{bdwI}
\end{equation}

\noi
for $j=0,1,\dots, J-1$. 
Moreover,
by
\eqref{PsiM}, \eqref{norm:Z},
and retaking time intervals, if necessary,
we have
\begin{equation}
\label{th est6}
\|\Psi\|_{Z(I_j)}
\le
\eta
\end{equation}
for $j=0,1,\dots, J-1$.

Fix $\tau>0$ to be chosen later.
Write $[t_0,t_0+\tau] = \bigcup_{j=0}^{J'} ([t_0,t_0+\tau] \cap I_j)$
for some $J' \le J-1$,
where
$[t_0,t_0+\tau] \cap I_j \neq \emptyset$ for $j \le J'$
and
$[t_0,t_0+\tau] \cap I_j = \emptyset$ for $j > J'$.

One can observe from above that the nonlinear flow $w$ is small on each $I_j$,
it follows that the linear evolution $S(t - t_j )w(t_j )$ is also small on each $I_j$.
From 
the Duhamel formula
\[
S(t - t_j )w(t_j )
=
w(t) + i
\int_{t_j}^t S(t-t') (|w|^4w)(t') dt',
\]

\noi
the nonlinear estimate \eqref{eq:non-est}
and \eqref{M-norm},
we have
\begin{align}
\notag
\begin{aligned}
    \|S(t-t_j)w(t_j)\|_{\wt X^1(I_j)}
    &\leq \|w\|_{\wt X^1(I_j)}
    +
    \big\| |w|^4w \big\|_{N^1(I_j)} \\
    &\le
    \eta+C \|w\|_{X^1(I_j)}^3 \|w\|_{\wt X^1(I_j)}^2
    \end{aligned}
\end{align}

\noi
for $j=0,1,\dots, J-1$.
By \eqref{bdWI1} and \eqref{bdwI},
we choose $0<\eta\ll1$ sufficiently small 
such that
\[
C \cdot C(R)^3 \eta^2 \leq \eta,
\]
which yields that
\begin{equation}
\label{th est4}
\|S(t-t_j)w(t_j)\|_{\wt X^1(I_j)}
\le
2\eta
\end{equation}
for $j=0,1,\dots, J-1$.

We now consider the estimate of $v$ on the first time interval 
$I_0$.
It follows from $v(t_0)=w(t_0)$ and \eqref{th est4} that
\begin{equation}
\label{th est5}
    \|S(t-t_0)v(t_0)\|_{\wt X^1(I_0)}
    =
    \|S(t-t_0)w(t_0)\|_{\wt X^1(I_0)}
    \le 2\eta.
\end{equation}

\noi
Let $\eta_0=\eta_0(R,M)>0$ be as in Proposition \ref{Prop:loc}.
By Proposition \ref{Prop:loc}
along with \eqref{PsiM}, \eqref{eq_asbd}, \eqref{th est6}, and \eqref{th est5},
we have
\begin{equation}
\begin{aligned}
\|v\|_{X^1(I_0)}
&\le
\|S(t-t_0) v(t_0) \|_{X^1(I_0)}
+
\|v-S(t-t_0) v(t_0) \|_{X^1(I_0)}
\\
&\le
R + C \eta
\le 2R
\label{vI0a}
\end{aligned}
\end{equation}

\noi
as long as 
$0<2\eta\leq\eta_0$.

Next, we estimate the error term $e$ defined in \eqref{pert}.
By using Proposition \ref{Prop:non est} with
\eqref{th est6} and \eqref{vI0a},
 we have
\begin{align*}
\| e \|_{N^1(I_0)}
&\les
|I_0|^\theta
\Big(
\|v\|_{X^1(I_0)}^2
\sum_{k=1}^3
\|v\|_{X^1(I_0)}^{3-k}\|\Psi\|_{Z(I_0)}^k
+ \| v \|_{X^1(I_0)} \| \Psi \|_{Z(I_0)}^4
+ \| \Psi \|_{Z(I_0)}^5
\Big)
\\
&\les
\tau^\theta
\big(
R^4 + 1
\big).
\end{align*}

\noi
Given $\eps>0$, 
by taking $\tau =\tau (R, \eps) \ll 1$,
we obtain
\begin{align*}
\| e \|_{N^1(I_0)}
\le\eps.
\end{align*}
In particular,
for $\eps < \eps_0$ with $\eps_0 = \eps_0(R) > 0$ 
dictated by Lemma \ref{Lem:Per}, the smallness condition \eqref{eq_smallness} is
satisfied on $I_0$.

Therefore, all the conditions of Lemma  \ref{Lem:Per} are satisfied on the first interval $I_0$, 
provided that $\tau = \tau (R)$ is chosen sufficiently small.
Thus, we have
\begin{align*}
    \|v\|_{X^1(I_0)}+\|w\|_{X^1(I_0)}\leq C_0(R),
    \quad \quad
    \|v-w\|_{X^1(I_0)}\leq C_0(R)\eps.
\end{align*}

\noi
In particular, with $C_1(R) := C_0(R)$,
we have
\begin{align}\label{th est7}
    \|v(t_1)-w(t_1)\|_{H^1}\leq C_1(R)\eps.
\end{align}

If $t_1 \ge \tau$,
we are done. 
If $t_1< \tau$,
we then use the iterative argument to move on the next time interval $I_1$.
By the triangle inequality,
Proposition \ref{Prop:linear}, 
\eqref{th est4}, and \eqref{th est7}, we derive 
\[
\begin{aligned}
    \|S(t-t_1)v(t_1)\|_{\wt X^1(I_1)}
    &\leq
    \|S(t-t_1)w(t_1)\|_{\wt X^1(I_1)}
    +\|S(t-t_1)(v-w)(t_1)\|_{\wt X^1(I_1)}\\
    &\le
    2\eta+\wt C_1(R)\eps
    \le 3\eta,
\end{aligned}
\]

\noi
where $\eps=\eps(R,\eta)>0$
sufficiently small such that 
$\wt C_1(R)\eps<\eta$.
From \eqref{PsiM}, \eqref{eq_asbd}, and \eqref{th est6},
we can apply Proposition \ref{Prop:loc} on $I_1$
as long as $3\eta\le\eta_0$.
Then, we have
\begin{equation}
\begin{aligned}
\|v\|_{X^1(I_1)}
&\le
\|S(t-t_1) v(t_1) \|_{X^1(I_1)}
+
\|v-S(t-t_1) v(t_1) \|_{X^1(I_1)}
\\
&\le
R + C \eta
\le 2R
\end{aligned}
\label{vI1a}
\end{equation}
as long as $3\eta\le\eta_0$.
Moreover,
it follows
from Proposition \ref{Prop:non est}
with \eqref{th est6}  and \eqref{vI1a}
that
\[
\| e \|_{N^1(I_1)}
\les
\tau^\theta
\big(
R^4+1\big)
\le\eps
\]
by choosing $\tau=\tau(R,\eps)>0$ 
sufficiently small.
Therefore, all the conditions of Lemma  \ref{Lem:Per} are satisfied on the interval $I_1$, 
provided that $\tau=\tau(\o, M,R, T, \eps)>0$ 
is chosen sufficiently small and that 
$(C_1(R)+1)\eps<\eps_0$.
Then, by Lemma \ref{Lem:Per},
we get
\begin{align*}
\|v-w\|_{X^1(I_1)}\leq C_0(R)(C_1(R)+1)\eps
=: C_2(R) \eps.
\end{align*}

\noi
In particular, we have
\begin{align*}
    \|v(t_2)-w(t_2)\|_{H^1}\leq C_2(R)\eps.
\end{align*}

Set
\[
C_j(R) := C_0(R) (C_{j-1}(R)+1)
\]
for $j=2,\dots,J'$.
We can argue inductively and obtain
\[
\| w(t_j) - v(t_j) \|_{H^1} \le C_j (R) \eps
\]
for all 
$j=0, \dots, J'$,
as long as
$\eta$, $\tau$, and $\eps$ satisfy
\[
(J'+2) \eta \le \eta_0,
\quad
(C_{J'}(R)+1) \eps < \eps_0,
\quad
\tau^\theta
\big(
R^4 + 1
\big)
\ll \eps.
\]
This guarantees existence of the solution $v$ to \eqref{SNLS2} on
$[t_0,t_0+\tau] \cap [0,T]$.
This completes the proof of Theorem \ref{Th}.
\end{proof}

\begin{ackno}
\rm 
The authors would like to thank Prof. Tadahiro Oh for suggesting this problem.
G.L.~was supported by The Maxwell Institute Graduate School in Analysis and its Applications, 
a Centre for Doctoral Training funded by the UK Engineering and Physical Sciences Research Council (Grant EP/L016508/01), 
the Scottish Funding Council, Heriot-Watt University and the University of Edinburgh;
by the European Research Council (grant no.~637995 ``ProbDynDispEq'' and grant no. 864138 ``SingStochDispDyn'');
and by the EPSRC New Investigator Award (grant no.~EP/S033157/1).
M.O. was supported by JSPS KAKENHI Grant number JP23K03182.
M.O. would like to thank the School of Mathematics at the University of Edinburgh for its hospitality, where this manuscript was prepared.
\end{ackno}


\begin{thebibliography}{99}
%


\bibitem{BOP15}
\'A. B\'enyi, T. Oh, O. Pocovnicu,
{\it On the probabilistic Cauchy theory of the cubic nonlinear Schr\"odinger equation on $\R^d$, $d\geq3$},
Trans. Amer. Math. Soc. Ser. B2 (2015), 1--50.



\bibitem{BO99}
J. Bourgain,
{\it Global wellposedness of defocusing critical nonlinear Schr\"odinger equation in the radial case},
J. Amer. Math. Soc. 12 (1999), no. 1, 145--171.



\bibitem{BLL24}
E. Brun, G. Li, R. Liu,
{\it Global well-posedness of the energy-critical stochastic nonlinear wave equations},
J. Differential Equations 397 (2024), 316--348.



\bibitem{BD03}
A. de Bouard, A. Debussche,
{\it The stochastic nonlinear Schr\"odinger equation in $H^1$},
Stochastic Anal. Appl. 21 (2003), no. 1, 97--126.





\bibitem{CM18}
K. Cheung, R. Mosincat, 
{\it Stochastic nonlinear Schr\"{o}dinger equations on tori}, 
Stoch. Partial Differ. Equ. Anal. Comput. 7 (2019), no. 2, 169--208.



\bibitem{CL22}
K. Cheung, G. Li,
{\it Global well-posedness of the 4-d energy-critical stochastic NLS with non-vanishing boundary condition},
Funkcial. Ekvac. 65 (2022), no. 3, 287--309.



\bibitem{CO1}
K. Cheung, O. Pocovnicu,
{\it Local well-posedness of stochastic nonlinear Schr\"{o}dinger equations on $\R^d$ with supercritical noise}, 
preprint.


\bibitem{CLO21}
K. Cheung, G. Li, T. Oh,
{\it  Almost conservation laws for stochastic nonlinear Schr\"odinger equations},
J. Evol. Equ. 21 (2021), no. 2, 1865--1894.



\bibitem{CKSTT08}
J. Colliander, M. Keel, G. Staffilani, H. Takaoka, T. Tao,
{\it Global well-posedness and scattering for the energy-critical nonlinear Schr\"odinger equation in $\R^3$},
Ann. of Math. 167 (2008), no. 3, 767--865.


%



\bibitem{FX19}
C. Fan, W. Xu,
{\it Subcritical approximations to stochastic defocusing mass-critical nonlinear Schr\"odinger equation on $\R$},
J. Differential Equations 268 (2019), no. 1, 160--185.



\bibitem{FX21}
C. Fan, W. Xu,
{\it Global well-posedness for the defocussing mass-critical stochastic nonlinear Schr\"odinger equation on $\R$ at $L^2$ regularity},
Anal. PDE 14 (2021), no. 8, 2561--2594.










\bibitem{HHK09}
M. Hadac, S. Herr, H. Koch,
{\it Well-posedness and scattering for the KP-II equation in a critical space},
Ann. Inst. H. Poincar\'{e} C Anal. Non Lin\'{e}aire 26 (2009), no. 3, 917--941.



\bibitem{HTT11}
S. Herr, D. Tataru, N. Tzvetkov,
{\it  Global well-posedness of the energy-critical nonlinear Schr\"odinger equation with small initial data in $H^{1}(\T^{3})  $},
Duke Math. J. 159 (2011), no. 2, 329--349.



\bibitem{HTT14}
S. Herr, D. Tataru, N. Tzvetkov,
{\it  Strichartz estimates for partially periodic solutions to Schr\"odinger equations in $4d$ and applications},
J. Reine Angew. Math. 690 (2014), 65--78.



\bibitem{IP12}
A. Ionescu, B. Pausader,
{\it   The energy-critical defocusing NLS on $\T^{3}   $},
Duke Math. J. 161 (2012), no. 8, 1581--1612.



\bibitem{IP12b}
A. Ionescu, B. Pausader,
{\it   Global well-posedness of the energy-critical defocusing NLS on $\R\times\T^3$},
Comm. Math. Phys. 312 (2012), no. 3, 781--831.


\bibitem{KM06}
C. Kenig, F. Merle,
{\it Global well-posedness, scattering and blow-up for the energy-critical, focusing, non-linear Schr\"odinger equation in the radial case},
Invent. Math. 166 (2006), no. 3, 645--675.





\bibitem{KT05}
H, Koch, D. Tataru,
{\it  Dispersive estimates for principally normal pseudodifferential operators},
 Comm. Pure Appl. Math. 58 (2005), no. 2, 217--284.






\bibitem{KV10}
R. Killip, M. Vi\c{s}an,
{\it The focusing energy-critical nonlinear Schr\"odinger equation in dimensions five and higher},
Amer. J. Math. 132 (2010), no. 2, 361--424.



\bibitem{KV12}
R. Killip, M. Vi\c{s}an,
{\it Global well-posedness and scattering for the defocusing quintic NLS in three dimensions},
Anal. PDE 5 (2012), no. 4, 855--885.




\bibitem{OO20}
T. Oh, M. Okamoto,
{\it On the stochastic nonlinear Schr\"odinger equations at critical regularities},
Stoch. Partial Differ. Equ. Anal. Comput. 8 (2020), no. 4, 869--894.


\bibitem{OOP19}
T. Oh, M. Okamoto, O. Pocovnicu,
{\it On the probabilistic well-posedness of the nonlinear Schr\"odinger equations with non-algebraic nonlinearities},
Discrete Contin. Dyn. Syst. 39 (2019), no. 6, 3479--3520.



\bibitem{RV07}
E. Ryckman, M. Vi\c{s}an,
{\it Global well-posedness and scattering for the defocusing energy-critical nonlinear Schr\"odinger equation in $\R^{1+4}$},
Amer. J. Math. 129 (2007), no. 1, 1--60.

\bibitem{St70}
E. M. Stein,
{\it Singular integrals and differentiability properties of functions},
Princeton Math. Ser., No. 30
Princeton University Press, Princeton, NJ, 1970, xiv+290 pp.


\bibitem{V07}
M. Vi\c{s}an,
{\it The defocusing energy-critical nonlinear Schr\"odinger equation in higher dimensions},
Duke Math. J. 138 (2007), no. 2, 281--374.


\bibitem{Y21}
H. Yue,
{\it Global well-posedness for the energy-critical focusing nonlinear Schr\"odinger equation on $\T^4$},
J. Differential Equations 280 (2021), 754--804.

\end{thebibliography}
\end{document}